\documentclass[12pt]{amsart}
\usepackage{amscd,amsmath,amsthm,amssymb}
\usepackage{mathtools}
\usepackage[margin=2cm]{geometry}
\usepackage{epsfig}
\usepackage{rawfonts}
\usepackage{enumerate}
\usepackage{graphics}
\usepackage{multirow}
\usepackage{xspace}
\usepackage{graphicx}
\usepackage{mathrsfs}
\usepackage{amsmath}
\usepackage{amsfonts}
\usepackage{amssymb}
\usepackage{amsthm}
\usepackage{graphicx}
\usepackage{booktabs}
\usepackage{caption}
\usepackage{listings}
\usepackage{setspace}
\usepackage[mathscr]{eucal}
\usepackage{pgfplots}
\usepackage{hyperref}
\usepackage{wrapfig}
\usepackage{floatflt,epsfig}
\usepackage{ dsfont }
\usepackage{amscd}
\usepackage{tikz-cd}
\usepackage{fancyhdr}
\usepackage[all]{xy}
\usepackage{latexsym}
\usepackage{amscd}
\usepackage{pifont}
\usepackage{subfig}
\usepackage{easyReview}
\usepackage{subfig}
\usepackage{pstricks-add}
\usepackage{pgf,tikz,pgfplots}
\pgfplotsset{compat=1.15}
\usepackage{mathrsfs}
\usetikzlibrary{arrows}
\usetikzlibrary[patterns]
 \usepackage[normalem]{ulem}
 \usepackage{multicol}

\def\Mc{{\mathcal M}}

\renewcommand{\qedsymbol}{$\square$}

\def\opn#1#2{\def#1{\operatorname{#2}}} 
\opn\chara{char} \opn\length{\ell} \opn\pd{pd} \opn\rk{rk}
\opn\projdim{proj\,dim} \opn\injdim{inj\,dim} \opn\rank{rank}
\opn\depth{depth} \opn\grade{grade} \opn\height{height}
\opn\embdim{emb\,dim} \opn\codim{codim}

\opn\Tr{Tr} \opn\bigrank{big\,rank}
\opn\superheight{superheight}\opn\lcm{lcm}
\opn\trdeg{tr\,deg}
	\opn\reg{reg} \opn\lreg{lreg} \opn\ini{in} \opn\lpd{lpd}
	\opn\size{size} \opn\sdepth{sdepth}
	\opn\link{link}\opn\fdepth{fdepth}\opn\lex{lex}\opn\dist{dist}
	\opn\div{div} \opn\Div{Div} \opn\cl{cl} \opn\Cl{Cl}
	\opn\Spec{Spec} \opn\Supp{Supp} \opn\supp{supp} \opn\Sing{Sing}
	\opn\Ass{Ass} \opn\Min{Min}\opn\Mon{Mon}
	\opn\Ann{Ann} \opn\Rad{Rad} \opn\Soc{Soc}
	\opn\Im{Im} \opn\Ker{Ker} \opn\Coker{Coker} \opn\Am{Am}
	\opn\Hom{Hom} \opn\Tor{Tor} \opn\Ext{Ext} \opn\End{End}
	\opn\Aut{Aut} \opn\id{id}
	
	\opn\nat{nat}
	\opn\pff{pf}
	\opn\Pf{Pf} \opn\GL{GL} \opn\SL{SL} \opn\mod{mod} \opn\ord{ord}
	\opn\Gin{Gin} \opn\Hilb{Hilb}\opn\sort{sort}
	\opn\aff{aff} \opn
	\con{conv} \opn\relint{relint} \opn\st{st}
	\opn\lk{lk} \opn\cn{cn} \opn\core{core} \opn\vol{vol}
	\opn\link{link} \opn\star{star}\opn\lex{lex}\opn\set{set}
	\opn\gr{gr}
	
	\def\pot#1#2{#1[\kern-0.28ex[#2]\kern-0.28ex]}

	\opn\dirlim{\underrightarrow{\lim}}
	\opn\inivlim{\underleftarrow{\lim}}
	\let\to=\rightarrow
	
	\def\Implies{\ifmmode\Longrightarrow \else
		\unskip${}\Longrightarrow{}$\ignorespaces\fi}
	\def\implies{\ifmmode\Rightarrow \else
		\unskip${}\Rightarrow{}$\ignorespaces\fi}
	\def\iff{\ifmmode\Longleftrightarrow \else
		\unskip${}\Longleftrightarrow{}$\ignorespaces\fi}

	\let\:=\colon
	\let\epsilon\varepsilon
	\let\kappa=\varkappa
	\def\qed{\ifhmode\textqed\fi
		\ifmmode\ifinner\quad\qedsymbol\else\dispqed\fi\fi}
	\def\textqed{\unskip\nobreak\penalty50
		\hskip2em\hbox{}\nobreak\hfil\qedsymbol
		\parfillskip=0pt \finalhyphendemerits=0}
	\def\dispqed{\rlap{\qquad\qedsymbol}}
	
	\opn\dis{dis}
	\def\pnt{{\raise0.5mm\hbox{\large\bf.}}}
	
	\opn\Lex{Lex}
	
        \newtheorem{Theorem}{Theorem}[section]
	\newtheorem{Lemma}[Theorem]{Lemma}
	\newtheorem{Corollary}[Theorem]{Corollary}
	
	\newtheorem{Remark}[Theorem]{Remark}
	
	\newtheorem{Example}[Theorem]{Example}
	
	\newtheorem{Definition}[Theorem]{Definition}

\begin{document}
   
        \title[ideals of $2$-minors]{Prime and Cohen--Macaulay ideals of $2$-minors with linear resolution}
	
\author[T.~Hibi]{Takayuki Hibi}
\address[Takayuki Hibi]
{Department of Pure and Applied Mathematics, 
Graduate School of Information Science and Technology, 
Osaka University, 
Suita, Osaka 565-0871, Japan}
\email{hibi@math.sci.osaka-u.ac.jp}

\author[A. A. Qureshi]{Ayesha Asloob Qureshi}      
   \address[Ayesha Asloob Qureshi]{Sabanci University, Faculty of Engineering and Natural Sciences, Orta Mahalle, Tuzla 34956, Istanbul, Turkey}	

\email{aqureshi@sabanciuniv.edu, ayesha.asloob@sabanciuniv.edu}

	\keywords{prime ideal, Cohen--Macaulay ring, linear resolution}

    \subjclass[2020]{13H10, 13P10}
 
	\thanks{} 

        \maketitle
      
\begin{abstract}

In the flame of combinatorics and toric ideals, we give a simple and elementary proof to a classification, which follows immediately from the classical Kronecker--Weierstrass theory, of prime and Cohen--Macaulay ideals of $2$-minors of matrices of variables with linear resolution.

\end{abstract}

\section*{Introduction}
Binomial ideals form an important class of ideals in commutative algebra,
with close connections to combinatorics and toric geometry. Among the
classical examples are determinantal ideals generated by the $2$-minors
of a matrix of variables. In particular, let
$S=K[x_1,\ldots,x_n,y_1,\ldots,y_n]$ be a polynomial ring over a field
$K$, and let
\[
X_0=
\begin{bmatrix}
x_1 & x_2 & \cdots & x_n\\
y_1 & y_2 & \cdots & y_n
\end{bmatrix}
\]
be the generic $2\times n$ matrix. Set
\[
J=I_2(X_0)
  =(x_i y_j-x_jy_i:1\leq i<j\leq n).
\]
It is well known that $J$ is prime, $S/J$ is Cohen--Macaulay, and $J$
has a linear resolution; see, for example, \cite{BrunsVetter}.

More generally, ideals of $2$-minors of matrices of linear forms
have been extensively studied in commutative algebra and algebraic
geometry. In particular, the classical Kronecker--Weierstrass
theory provides a structural description of such matrices in
terms of scroll, Jordan, and nilpotent blocks; see, for example,
\cite{RS}. The ideals of $2$-minors of concatenations of scroll
blocks define rational normal scrolls, whose algebraic and
homological properties are well understood. Their toric
structure and Gr\"obner bases have also been studied;
see, for example, \cite{Pet}.

In our previous work \cite{HQS}, we studied quadratic binomial
ideals from a converse point of view. Namely, we asked to what
extent a prime quadratic binomial ideal with Cohen--Macaulay
quotient and linear resolution must arise as the ideal of
$2$-minors of a matrix of variables. This question was
motivated in part by several families of binomial ideals arising
from combinatorial structures for which these three properties
lead naturally to determinantal ideals. In \cite{HQS}, we
obtained partial results in this direction and posed the related
problem of classifying matrices of variables whose ideals of
$2$-minors satisfy these properties.

The purpose of the present paper is to provide a direct
combinatorial approach to this classification problem for
$2\times n$ matrices of not necessarily distinct variables.
Rather than appealing to the general Kronecker--Weierstrass
theory, we work directly with the pattern of repeated variables.
Our approach yields a simple combinatorial characterization of
primeness, together with an explicit toric parametrization and
a quadratic Gr\"obner basis. In particular, it provides an
elementary and constructive description of these determinantal
ideals in terms of the original variables.

Let $X$ be a $2\times n$ matrix of not necessarily distinct
variables in $S$. We call $X$ a \emph{PCML matrix} if
(i) $I_2(X)$ is prime,
(ii) $S/I_2(X)$ is Cohen--Macaulay, and
(iii) $I_2(X)$ has a linear resolution.
Our aim is to characterize such matrices in combinatorial terms.

In Section~\ref{sec:1}, for each subset $\delta\subseteq\{2,\ldots,n\}$,
we introduce the matrix
\[
X_n(\delta)=
\begin{pmatrix}
x_1 & x_2 & \cdots & x_n\\
\xi_1 & \xi_2 & \cdots & \xi_n
\end{pmatrix},
\qquad
\xi_i=
\begin{cases}
y_i, & i\notin\delta,\\
x_{i-1}, & i\in\delta.
\end{cases}
\]
This family contains both the generic matrix, when
$\delta=\emptyset$, and the Hankel matrix, when
$\delta=\{2,\ldots,n\}$. In Theorem~1.1, we establish that
every $X_n(\delta)$ is a PCML matrix. Our proof of primeness
is based on an explicit toric parametrization. More precisely,
we construct a homomorphism into a Laurent polynomial ring
whose kernel is precisely $I_2(X_n(\delta))$. We also describe
the initial ideal with respect to a suitable reverse
lexicographic order, thereby obtaining an explicit quadratic
Gr\"obner basis for $I_2(X_n(\delta))$.

In Section~\ref{sec:2}, we associate with a matrix
\[
X=
\begin{pmatrix}
x_1 & x_2 & \cdots & x_n\\
z_1 & z_2 & \cdots & z_n
\end{pmatrix},
\]
where the entries in each row are pairwise distinct and
$z_i\neq x_i$ for every $i$, a directed graph $G_X$ on
the vertex set $\{1,\ldots,n\}$ such that $i\longrightarrow j$ if and only if $x_i=z_j.$ As observed in Remark~\ref{rem:def}, every connected component of
$G_X$ is either a directed path or a directed cycle.
In Theorem~\ref{prime}, we prove that $I_2(X)$ is prime if and
only if $G_X$ contains no directed cycle. Moreover, in
the acyclic case, the directed paths determine a permutation
of the columns and a relabeling of the variables that
transform $X$ into a matrix of the form $X_n(\delta)$.
Combining this with Theorem~\ref{prime+CM+linear}
yields our main classification, Theorem~\ref{mainPCML}:
\[
X\text{ is PCML}
\quad\Longleftrightarrow\quad
I_2(X)\text{ is prime}
\quad\Longleftrightarrow\quad
G_X\text{ contains no directed cycle}.
\]
Thus, within this class of matrices, primeness alone
implies the Cohen--Macaulay property and linear resolution.
Although this conclusion is consistent with the classical
theory of rational normal scrolls, our approach gives a
direct combinatorial criterion and an elementary proof
without requiring Kronecker--Weierstrass reduction.

Theorem~\ref{mainPCML} also naturally raises the question whether
the remaining two properties can be characterized
combinatorially without assuming primality. More precisely,
for a matrix $X$ as above, it would be interesting to
determine, in terms of the directed graph $G_X$, when
$S/I_2(X)$ is Cohen--Macaulay and when $I_2(X)$ has a
linear resolution. In particular, one may ask how the
structure of the directed cycles of $G_X$ influences
these two properties. A broader direction is to ask
whether an analogous combinatorial characterization
can be obtained for $m\times n$ matrices with $m>2$.
In that setting, the possible patterns of repeated
variables are substantially more complicated, and
the simple directed-graph description used here is
no longer immediately available. Understanding to what extent our combinatorial approach
admits a higher-dimensional analogue therefore appears
to be a natural problem for further study.

\section{Ideals of $2$-minors of $2 \times n$ matrices}\label{sec:1}
Let $S=K[x_1,\ldots,x_n,y_1,\ldots,y_n]$ denote the polynomial ring in $2n$ variables over a field $K$.  Fix $\delta \subset \{2, \ldots, n\}$ and introduce a $2 \times n$ matrix
$$
X_n(\delta)=
\begin{bmatrix}
   x_1 & x_2 & \cdots & x_n \\
   \xi_{1} & \xi_{2} & \cdots & \xi_{n} 
\end{bmatrix}
$$
where 
\[
\xi_i
=
\begin{cases}
\,y_i, & \text{if} \quad i \not\in \delta,\\[2pt]
\,x_{i-1},     & \text{if} \quad i \in \delta.
\end{cases}
\]
In particular, 
\[
X_n(\{2,\ldots, n\}) \, = \, 
\begin{bmatrix}
   x_1 & x_2 & x_3 & \cdots & x_n \\
   y_{1} & x_1 & x_2 & \cdots & x_{n-1}  
\end{bmatrix},
\qquad
X_n(\emptyset)=
\begin{bmatrix}
   x_1 & x_2 & \cdots & x_n \\
   y_{1} & y_{2} & \cdots & y_n  
\end{bmatrix}.
\]
Let $S_\delta$ denote the polynomial ring in the variables $x_1, \ldots, x_n$ and $y_j$ with $j \not\in \delta$ over $K$ and $I_n(\delta)$ the ideal of $S_\delta$ generated by all $2$-minors of $X_n(\delta)$.

The two extreme cases of this family are classical. If
$\delta=\emptyset$, then $I_n(\delta)$ is the ideal of maximal minors
of the generic $2\times n$ matrix. It is well known that this ideal is
prime, its quotient ring is Cohen--Macaulay, and it has a linear
resolution; see, for example, \cite{BrunsVetter}. At the other extreme,
if $\delta=\{2,\ldots,n\}$, then, after relabeling the variables,
$I_n(\delta)$ is the ideal generated by the $2$-minors of a
$2\times n$ Hankel matrix. This is the defining ideal of a rational
normal curve; in particular, it is prime, its quotient ring is
Cohen--Macaulay, and its minimal free resolution is given by the
Eagon--Northcott complex, hence the ideal has a linear resolution;
see \cite{ConcaRNC}.

The following theorem shows that these three properties persist for
every intermediate choice of
$\delta\subseteq\{2,\ldots,n\}$, thereby interpolating between the
generic and Hankel cases.

\begin{Theorem}\label{prime+CM+linear}
One has:
\begin{enumerate}
\item[(i)] the ideal $I_n(\delta)\subset S_\delta$ is prime;
\item[(ii)] the quotient ring $S_\delta/I_n(\delta)$ is Cohen--Macaulay;
\item[(iii)] $I_n(\delta)$ has a linear resolution.
\end{enumerate}
\end{Theorem}

We begin by proving parts (ii) and (iii) of
Theorem~\ref{prime+CM+linear}. The idea is to obtain
$I_n(\delta)$ from the generic determinantal ideal $I_n(\emptyset)$
by imposing the linear relations
\[
x_{i-1}-y_i,\qquad i\in\delta.
\]
We first record the elementary identities that make this reduction precise. Let $i\notin\delta$ and $j\in\delta$. Then
\[
(x_i x_{j-1}-x_jy_i,\,x_{j-1}-y_j)
=
(x_i y_j-x_jy_i,\,x_{j-1}-y_j).
\]
Similarly, if $i,j\in\delta$ with $i\neq j$, then
\begin{eqnarray*}
&&
(x_i x_{j-1}-x_jx_{i-1},\,
  x_{i-1}-y_i,\,
  x_{j-1}-y_j)
\\
&=&
(x_i y_j-x_jy_i,\,
  x_{i-1}-y_i,\,
  x_{j-1}-y_j).
\end{eqnarray*}
It follows that
\[
S/(I_n(\emptyset),\{x_{i-1}-y_i:i\in\delta\})
=
S/(I_n(\delta),\{x_{i-1}-y_i:i\in\delta\}).
\]

\begin{Lemma}
\label{regular_sequence}
One has
\[
S/(I_n(\emptyset),\{x_{i-1}-y_i:i\in\delta\})
\cong S_\delta/I_n(\delta).
\]
\end{Lemma}

\begin{proof}
Modulo the relations $x_{i-1}-y_i$, $i\in\delta$, each variable
$y_i$ with $i\in\delta$ is identified with $x_{i-1}$. Hence
\[
S/(\{x_{i-1}-y_i:i\in\delta\})\cong S_\delta.
\]
Under this identification, the image of $I_n(\emptyset)$ is precisely
$I_n(\delta)$, by the identities above. The assertion follows.
\end{proof}

We can now deduce the Cohen--Macaulay property and linearity of the
resolution from the corresponding properties of the generic
determinantal ideal.

\begin{Corollary}
\label{CM_linear}
The quotient ring $S_\delta/I_n(\delta)$ is Cohen--Macaulay and
$I_n(\delta)$ has a linear resolution.
\end{Corollary}

\begin{proof}
Recall that $S/I_n(\emptyset)$ is Cohen--Macaulay and
$I_n(\emptyset)$ has a linear resolution. Since
\[
\{x_{i-1}-y_i:i\in\delta\}
\]
is a regular sequence of $S/I_n(\emptyset)$
(\cite[Theorem~4.1]{Eis}), Lemma~\ref{regular_sequence} implies that
$S_\delta/I_n(\delta)$ is Cohen--Macaulay. Furthermore, the $h$-vector of $S_\delta/I_n(\delta)$ coincides with
that of $S/I_n(\emptyset)$. Since $I_n(\emptyset)$ has a linear
resolution, it follows that $I_n(\delta)$ has a linear resolution.
\end{proof}

We now turn to a proof of (i) of Theorem \ref{prime+CM+linear}.
Let $<_{\rm rev}$ denote the reverse lexicographic order of $S$ induced by the ordering of the variables 
\[
y_1 < x_1 < y_2 < x_2 < y_3 < \cdots < y_{n} < x_n
\]
and $<_{\rm rev}^*$ the restriction of $<_{\rm rev}$ on $S_\delta$.  Let ${\rm in}_{<_{\rm rev}^*}(I_n(\delta))$ denote the initial ideal of $I_n(\delta)$ with respect to $<_{\rm rev}^*$.  Let $\Omega=A \cup B$, where 
    \begin{eqnarray*}
&A=\{x_k y_\ell:1 \leq k \leq n, \ell \not\in \delta, k < \ell\},& \\
&B=\{x_k x_{\ell-1}:1 \leq k \leq n, \ell \in \delta, k < \ell\}.&
    \end{eqnarray*}
and $(\Omega) \subset S_\delta$ the monomial ideal generated by those monomials belonging to $\Omega$.
\begin{Lemma}
    One has $(\Omega) \subset {\rm in}_{<_{\rm rev}^*}(I_n(\delta))$.
\end{Lemma}

Suppose that $2 \not\in \delta$ and define the ring homomorphism $$\varphi:S_\delta \to K[y_1,x_1^{-1}, \{x_i: i+1 \not\in \delta\}]$$ by setting
\begin{itemize}
    \item 
    $\varphi(y_1)=y_1$,
    \item 
    $\varphi(x_i)=x_i, \, i+1 \not\in \delta$,
    \item 
    $\varphi(y_j)=(y_1/x_1)^{i_0-j+1}x_{i_0}, \, \text{where $1 < j \not\in \delta$ and $i_0$ is the smallest $i$ with $i+1 \not\in \delta$ and  $j \leq i$}$,
    \item
    $\varphi(x_j)=(y_1/x_1)^{i_0-j}x_{i_0}, \, \text{where $j+1 \in \delta$ and $i_0$ is the smallest $i$ with $i+1 \not\in \delta$ and  $j < i$}$.
\end{itemize}

For example, if $n=7$ and $\delta=\{3,4,6\}$, then
\begin{eqnarray*}
&&\varphi\left(
\begin{bmatrix}
   x_1 & x_2 & x_3 & x_4 & x_5 & x_6 & x_7 \\
   y_1 & y_2 & x_2 & x_3 & y_5 & x_5 & y_7  
\end{bmatrix}
\right) \\
&&=
\begin{bmatrix}
   x_1 & (y_1/x_1)^2x_4 & (y_1/x_1)x_4 & x_4 & (y_1/x_1)x_6 & x_6 & x_7 \\   y_1 & (y_1/x_1)^3x_4 & (y_1/x_1)^2x_4 & (y_1/x_1)x_4 & (y_1/x_1)^2x_6 & (y_1/x_1)x_6 & (y_1/x_1)x_7  
\end{bmatrix}.
\end{eqnarray*}

One has 
\[
I_n(\delta) \subset {\rm Ker}(\varphi), \quad {\rm in}_{<_{\rm rev}^*}(I_n(\delta)) \subset {\rm in}_{<_{\rm rev}^*}({\rm Ker}(\varphi)).
\]
Let $\Mc_\delta$ denote the set of monomials of $S_\delta$ and 
\[
\Omega^\# := \{u \in \Mc_\delta : u \not\in (\Omega)\}.
\]
\begin{Lemma}
    \label{linear_independent}
The set of Laurent monomials $\{\varphi(u) : u \in \Omega^\# \}$ is linearly independent. 
\end{Lemma}

\begin{proof}
Let $\supp{(u)}$ denote the support of a monomial $u \in S_\delta$.  Our task is to prove that if $u,v \in \Omega^\#$ and if $u$ and $v$ are relatively prime, then $\varphi(u) \neq \varphi(v)$.  Set 
\[
V_\delta := \{x_i : i+1 \not\in \delta\} \cup \{y_i : i \not\in \delta\}.
\]
Let
\[
\{x_i : i+1 \not\in \delta\} = \{x_1, x_{i_1}, \ldots, x_{i_{n-|\delta|-1}}\}, \quad 1 < i_1 < \cdots < i_{n-|\delta|-1}
\]
and
\[
\{y_i : i \not\in \delta\} = \{y_1, y_{j_1}, \ldots, y_{j_{n-|\delta|-1}}\}, \quad 1 < j_1 < \cdots < j_{n-|\delta|-1}.
\]
One has 
\[
\varphi(y_{j_k}) = (y_1/x_1)^{i_k - j_k + 1} x_{i_k}, \quad k = 2,\ldots, n-|\delta|-1.
\]

\medskip

\noindent
{\bf (First Step)}  If $\supp(u) \cup \supp(v) \subset V_\delta$, then the desired result follows from \cite[p.~98]{Hibi}.   



\medskip

\noindent
{\bf (Second Step)}  Let $\supp(u) \subset V_\delta$ and $\supp(v) \not\subset V_\delta$.  Let $i_0$ denote the unique integer $i$ for which $x_{i} \in \supp(v)$ and $i+1 \in \delta$.  One has 
\[
\supp(v) \subset \{x_{i_0}\} \cup \{x_i \in V_\delta: i > i_0\} \cup \{y_j \in V_\delta: j \leq i_0 \} .
\]
Let $i^* > i_0$ denote the smallest integer $i$ for which $x_i \in V_\delta$ and $i_* < i_0$ the biggest integer $i$ for which $y_i \in V_\delta$.  
Let $x_{i^*} \in \supp(u)$.  Since $y_j \not\in \supp(u)$ for $j > i_0$, it follows that $\varphi(u) \neq \varphi(v)$. 

Let $x_{i^*} \not\in \supp(u)$.  Then $y_{i_*} \in \supp(u)$.  If $\supp(u) \not\subset \{y_i : i \not\in \delta \}$, then clearly $\varphi(u) \neq \varphi(v)$.  If $\supp(u) \subset \{y_i : i \not\in \delta \}$, then considering the powers of $y_1/x_1$, one has $\varphi(u) \neq \varphi(v)$. 

\medskip

\noindent
{\bf (Third Step)}  Let $\supp(u) \not\subset V_\delta$ and $\supp(v) \not\subset V_\delta$.  Let $i_u$ (resp. $i_v$) denote the unique integer $i$ for which $x_{i} \in \supp(u)$ (resp. $x_{i} \in \supp(v)$) and $i+1 \in \delta$.  Let $i_u < i_v$.

If there is no $y_j$ with $i_u < j < i_v$, then, since 
\[
(\supp(u) \cup \supp(v) )\cap V_\delta \subset \{x_i \in V_\delta : i > i_v\} \cup \{y_j : j \leq i_u\}, 
\]
one has $\varphi(u) \neq \varphi(v)$.

Suppose that there is $y_j$ with $i_u < j < i_v$.  Let $\varphi(x_{i_u}) = (y_1/x_1)^a x_{i_*}$ and $\varphi(x_{i_v}) = (y_1/x_1)^b x_{i^*}$, where $i_* < i^*$.  Let $\varphi(y_{j_*}) = (y_1/x_1)^{a'} x_{i_*}$ and $\varphi(y_{j^*}) = (y_1/x_1)^{b'} x_{i^*}$, where $j_* < j^*$.  Now, since $x_{i^*} \in \supp(u)$ and $y_{j_*} \in \supp(v)$,  considering the powers of $y_1/x_1$, one has $\varphi(u) \neq \varphi(v)$. 
\, \, \, 
\end{proof}

When $2 \in \delta$, since $x_1 \not\in V_\delta$, slightly modifying the map $\varphi$ is required.  Let $q \geq 2$ denote the smallest integer $i$ for which $x_i \in V_\delta$.  We then define $\varphi: S_\delta \to K[y_1, x_{q}^{-1}, \{x_i : i+1 \not\in \delta\}]$ by starting with $\varphi(y_1) = y_1^{q}, \varphi(x_1) = x_{q}y_1^{q-1}$ and $\varphi(x_i) = x_i^{q}$ if $x_i \in V_\delta$.  More precisely,  

\begin{itemize}
    \item 
    $\varphi(y_1)=y_1^q$,
    \item 
    $\varphi(x_1) = x_{q}y_1^{q-1}$,
    \item 
    $\varphi(x_i)=x_i^q, \, i+1 \not\in \delta$,
    \item 
    $\varphi(y_j)=(y_1/x_q)^{i_0-j+1}x_{i_0}^q, \, \text{where $1 < j \not\in \delta$ and $i_0$ is the smallest $i$ with $i+1 \not\in \delta$ and  $j \leq i$}$,
    \item
    $\varphi(x_j)=(y_1/x_q)^{i_0-j}x_{i_0}^q, \, \text{where $j+1 \in \delta$ and $i_0$ is the smallest $i$ with $i+1 \not\in \delta$ and  $j < i$}$.
\end{itemize}
Our proof of Lemma \ref{linear_independent} can be valid with slight modifications.

For example, if $n=5$ and $\delta=\{2,4\}$, then $q=2$ and
\begin{eqnarray*}
&&\varphi\left(
\begin{bmatrix}
   x_1 & x_2 & x_3 & x_4 & x_5  \\
   y_1 & x_1 & y_3 & x_3 & y_5   
\end{bmatrix}
\right) \\
&&=
\begin{bmatrix}
   x_2 y_1& x_2^2 & (y_1/x_2)x_4^2 & x_4^2 & x_5^2  \\   
   y_1^2 & x_2 y_1 & (y_1/x_2)^2x_4^2 & (y_1/x_2)x_4^2 & (y_1/x_2)x_5^2   
\end{bmatrix}.
\end{eqnarray*}

Now, a standard fact on Gr\"obner basis \cite[Proposition 2.2.5]{HHgtm260} guarantees that

\begin{Corollary}
\label{toric_ideal}
One has
\[
(\Omega)
= {\rm in}_{<_{\rm rev}^*}(I_n(\delta))
= {\rm in}_{<_{\rm rev}^*}({\rm Ker}(\varphi))
\]
and $I_n(\delta)={\rm Ker}(\varphi)$.
In particular, $I_n(\delta)$ is prime, and the
$2$-minors of $X_n(\delta)$ form a quadratic
Gr\"obner basis with respect to $<_{\rm rev}^*$.
\end{Corollary}

\section{Classification of PCML matrices}\label{sec:2}

Let $X$ be a $2\times n$ matrix of not necessarily distinct
variables belonging to a polynomial ring $S$ over a field $K$,
and let $I$ be the ideal generated by all $2$-minors of $X$.
We exclude the trivial case $I=0$. Since removing repeated
identical columns does not change the ideal of $2$-minors,
we may assume that the columns of $X$ are pairwise distinct.
If a variable appears more than once in the same row, or
if the two entries of some column are equal, then a nonzero
$2$-minor factors as a product of two linear forms. Since
$I$ contains no nonzero linear forms, it cannot be prime.
Therefore, in classifying PCML matrices, we may restrict
our attention to matrices whose entries in each row are
pairwise distinct and whose two entries in every column
are distinct. After relabeling the variables, we may thus write
\[
X=
\begin{bmatrix}
x_1 & x_2 & \cdots & x_n\\
z_1 & z_2 & \cdots & z_n
\end{bmatrix},
\]
where $z_1,\ldots,z_n$ are pairwise distinct elements of
$\{x_1,\ldots,x_n,y_1,\ldots,y_n\}$ and $z_i\neq x_i$ for
$i=1,\ldots,n$.

Toward the classification of PCML matrices, we first characterize those matrices whose ideals of $2$-minors are prime. For this purpose, we associate a directed graph to $X$. We first recall the graph-theoretic terminology that will be used.

Let $G$ be a directed graph. A \emph{directed path} in $G$ is a sequence of distinct vertices \[ v_1,v_2,\ldots,v_r \] such that $(v_i,v_{i+1})\in E(G)$ for each $i=1,\ldots,r-1$. In particular, an isolated vertex is regarded as a directed path. A \emph{directed cycle} of length $r$ in $G$ consists of a directed path
$v_1,v_2,\ldots,v_r$, with $r\geq 2$, which also satisfies
$(v_r,v_1)\in E(G)$. In particular, if $(u,v)$ and
$(v,u)$ belong to $E(G)$, then $u,v$ form a directed cycle of length $2$.

\begin{Definition}
{\em
Let
\[
X=
\begin{bmatrix}
x_1 & x_2 & \cdots & x_n\\
z_1 & z_2 & \cdots & z_n
\end{bmatrix},
\]
where $z_1,\ldots,z_n$ are pairwise distinct elements of
$\{x_1,\ldots,x_n,y_1,\ldots,y_n\}$ and $z_i\neq x_i$ for all
$i=1,\ldots,n$.
We associate to $X$ a directed graph $G_X$ on the vertex set
$\{1,\ldots,n\}$ as follows: for $1\le i,j\le n$, there is a directed edge
from $i$ to $j$ if and only if $x_i=z_j$.
}
\end{Definition}

For the directed graph $G_X$ associated with the matrix $X$, directed loops cannot occur, since a loop at the vertex $i$ would mean $x_i=z_i$, contrary to the assumption $z_i\neq x_i$. Hence $G_X$ has no directed cycles of length $1$.

The following remark describes the structure of the connected components of the graph $G_X$.

\begin{Remark}\label{rem:def}
{\em The graph $G_X$ has a particularly simple structure. By construction, every
vertex of $G_X$ has indegree at most one. Since the entries
$z_1,\ldots,z_n$ are pairwise distinct, every vertex has outdegree at most
one. Hence every connected component of $G_X$ is either a directed path or
a directed cycle.}
\end{Remark}

The following example illustrates the construction of the graph $G_X$.

\begin{Example}\label{example}
{\em 
(1)
Consider the matrix
\[
X=
\begin{bmatrix}
x_1 & x_2 & x_3 & x_4 & x_5\\
x_2 & x_1 & y_1 & x_3 & y_5
\end{bmatrix}.
\]
Then $G_X$ has vertex set $\{1,2,3,4,5\}$ and directed edges
\[
1\to2,\qquad
2\to1,\qquad
3\to4,
\]
since $x_1=z_2$, $x_2=z_1$, and $x_3=z_4$. Thus $G_X$ has three connected
components: a directed cycle on the vertices $\{1,2\}$, a directed path $3\to4$, and the isolated vertex $5$.  If $I$ denotes the ideal generated by the $2$-minors of $X$, then
\[
x_1^2-x_2^2=(x_1-x_2)(x_1+x_2)\in I,
\]
while neither factor belongs to $I$. Hence $I$ is not prime. 

(2) Consider the matrix
\[
Y=
\begin{bmatrix}
x_1 & x_2 & x_3 & x_4 & x_5 & x_6\\
z_1 & z_2 & x_1 & z_4 & x_2 & x_3
\end{bmatrix},
\]
where $z_1,z_2,z_4$ are distinct variables in
$\{y_1,\ldots,y_6\}$. Then $G_Y$ has three connected components,
namely the directed paths
\[
1\to3\to6,\qquad
2\to5,
\]
and the isolated vertex $4$. Relabeling the variables by setting
\[
z_1=y_1,\qquad
z_2=y_3,\qquad
z_4=y_6,
\]
and then reordering the columns according to the connected components,
we obtain the matrix
\[
\begin{bmatrix}
x_1 & x_3 & x_6 & x_2 & x_5 & x_4\\
y_1 & x_1 & x_3 & y_3 & x_2 & y_6
\end{bmatrix}.
\]
Hence, after relabeling the variables and permuting the columns of $Y$,
the resulting matrix is precisely $X_6(\delta)$, where
$\delta=\{2,3,5\}$. Consequently, if $J$ denotes the ideal generated by
the $2$-minors of $Y$, then $J$ is identified with the ideal
$I_6(\delta)$.
}
\end{Example}

The next lemma isolates the obstruction to primeness arising from a
directed cycle and will be used in the proof of Theorem~\ref{prime}.

\begin{Lemma}\label{lemma:cycle}
Let $R=K[u_1,\ldots,u_r]$ be a polynomial ring and
\[
U=
\begin{bmatrix}
u_1 & u_2 & \cdots & u_r\\
u_2 & u_3 & \cdots & u_1
\end{bmatrix}.
\]
Then the ideal generated by the $2$-minors of $U$ is not prime.
\end{Lemma}

\begin{proof}
Let $J$ be the ideal generated by the $2$-minors of $U$. Throughout the proof, indices are taken modulo $r$, in particular, $u_{r+1}=u_1$. We first claim that
\[
u_1^r-u_2^r\in J.
\]
Indeed, for $k=1,\ldots,r$, the minor of the first and $k$-th columns gives
\begin{equation}\label{eq1}
u_1u_{k+1}-u_2u_k\in J.
\end{equation}
We prove by induction on $k$ that
\begin{equation}\label{eq:ind}
u_1^{k-1}u_{k+1}-u_2^k\in J
\end{equation}
for all $k=1,\ldots,r$.

For $k=1$, the assertion is trivial. Assume that $k\geq 2$ and that the assertion holds for $k-1$. Multiplying \eqref{eq1} by $u_1^{k-2}$ gives
\begin{equation}\label{eq2}
u_1^{k-1}u_{k+1}-u_1^{k-2}u_2u_k\in J.
\end{equation}
By the induction hypothesis, $u_1^{k-2}u_k-u_2^{k-1}\in J.$ Multiplying this by $u_2$, we obtain
\begin{equation}\label{eq3}
u_1^{k-2}u_2u_k-u_2^k\in J.
\end{equation}
Adding \eqref{eq2} and \eqref{eq3}, we get
\[
u_1^{k-1}u_{k+1}-u_2^k\in J.
\]
This proves \eqref{eq:ind}. Taking $k=r$ in \eqref{eq:ind}, and using $u_{r+1}=u_1$, we obtain
\[
u_1^r-u_2^r\in J.
\]
This proves the claim.

Now write
\[
u_1^r-u_2^r
=
(u_1-u_2)
\left(
u_1^{r-1}+u_1^{r-2}u_2+\cdots+u_2^{r-1}
\right).
\]
We next prove that
\[
u_1^{r-1}+u_1^{r-2}u_2+\cdots+u_2^{r-1}
\equiv
u_r^{r-2}(u_1+u_2+\cdots+u_r)
\pmod J.
\]
To prove this congruence, note that the minor of the $k$-th and $r$-th
columns yields $u_1u_k-u_ru_{k+1}\in J$ for $k=1,\ldots,r$. Hence
\begin{equation}\label{eq:colr}
u_1u_k
\equiv
u_ru_{k+1}
\pmod J
\qquad \text{for } k=1,\ldots,r.
\end{equation}
Similarly, the minor of the first and $k$-th columns yields
$u_1u_{k+1}-u_2u_k\in J$ for $k=1,\ldots,r$. Hence
\begin{equation}\label{eq:col1}
u_2u_k
\equiv
u_1u_{k+1}
\pmod J
\qquad \text{for } k=1,\ldots,r.
\end{equation}
We first claim that, for $i=0,\ldots,r-1$,
\begin{equation}\label{eq:claim}
u_1^{r-1-i}u_2^i
\equiv
u_1^{r-2}u_{i+1}
\pmod J.
\end{equation}
For $i=0$, the claim is immediate. Let $i\geq 1$. Repeated use of
\eqref{eq:col1} yields
\[
u_2^i
\equiv
u_1^{i-1}u_{i+1}
\pmod J.
\]
Multiplying by $u_1^{r-1-i}$, we obtain
\[
u_1^{r-1-i}u_2^i
\equiv
u_1^{r-2}u_{i+1}
\pmod J.
\]
Next, repeated application of \eqref{eq:colr} gives
\[
u_1^{r-2}u_{i+1}
\equiv
u_1^{r-2-t}u_r^t u_{i+1+t}
\pmod J
\]
for every $t=0,\ldots,r-2$. Taking $t=r-2$, we obtain
\[
u_1^{r-2}u_{i+1}
\equiv
u_r^{r-2}u_{i+r-1}
\pmod J.
\]
Since the indices are taken modulo $r$, we have
$u_{i+r-1}=u_{i-1}$. Therefore
\[
u_1^{r-2}u_{i+1}
\equiv
u_r^{r-2}u_{i-1}
\pmod J.
\]
Combining this with \eqref{eq:claim}, we obtain
\[
u_1^{r-1-i}u_2^i
\equiv
u_r^{r-2}u_{i-1}
\pmod J
\qquad \text{for } i=0,\ldots,r-1.
\]
As $i$ runs from $0$ to $r-1$, the indices $i-1$ run through all
residue classes modulo $r$. Hence the right-hand side runs through the
monomials
\[
u_r^{r-2}u_1,\,
u_r^{r-2}u_2,\,
\ldots,\,
u_r^{r-2}u_r.
\]
Summing over $i=0,\ldots,r-1$, we obtain
\[
u_1^{r-1}+u_1^{r-2}u_2+\cdots+u_2^{r-1}
\equiv
u_r^{r-2}(u_1+u_2+\cdots+u_r)
\pmod J.
\]
Since $u_1^r-u_2^r\in J$, it follows that
\[
(u_1-u_2)u_r^{r-2}(u_1+\cdots+u_r)\in J.
\]
However, $J$ is a homogeneous ideal generated by quadratic forms. Hence
$J$ contains no nonzero linear forms. In particular,
\[
u_1-u_2\notin J, \quad u_r\notin J
\qquad\text{and}\qquad
u_1+\cdots+u_r\notin J.
\]
If $J$ were prime, it would follow that one of
$u_1-u_2$, $u_r$, or $u_1+\cdots+u_r$ belongs to $J$, a contradiction.
Hence $J$ is not prime.
\hspace{15.1cm}
\end{proof}

\begin{Theorem}
\label{prime}
Let
\[
X=
\begin{bmatrix}
x_1 & x_2 & \cdots & x_n\\
z_1 & z_2 & \cdots & z_n
\end{bmatrix},
\]
where $z_1,\ldots,z_n$ are pairwise distinct elements of $\{x_1,\ldots,x_n,y_1,\ldots,y_n\}$ and $z_i\neq x_i$ for all $i=1,\ldots,n$. Then the ideal of $2$-minors of $X$ is prime if and only if $G_X$ contains no directed cycle.
\end{Theorem}

\begin{proof}
Suppose that $G_X$ contains a directed cycle $i_1\to i_2\to\cdots\to i_r\to i_1.$ Then, by definition,
\[
x_{i_1}=z_{i_2},\quad
x_{i_2}=z_{i_3},\quad
\ldots,\quad
x_{i_r}=z_{i_1}.
\]
Hence, after interchanging the two rows and cyclically
permuting the columns, the corresponding submatrix
of $X$ is
\[
\begin{bmatrix}
x_{i_1} & x_{i_2} & \cdots & x_{i_r}\\
x_{i_2} & x_{i_3} & \cdots & x_{i_1}
\end{bmatrix}.
\]
More precisely, the proof of Lemma~2.4 shows that
\[
(x_{i_1}-x_{i_2})x_{i_r}^{\,r-2}
(x_{i_1}+x_{i_2}+\cdots+x_{i_r})\in I.
\]
Since $I$ is a homogeneous ideal generated by quadratic forms, it
contains no nonzero linear forms. In particular,
\[
x_{i_1}-x_{i_2}\notin I,\qquad
x_{i_r}\notin I,\qquad\text{and}\qquad
x_{i_1}+\cdots+x_{i_r}\notin I.
\]
If $I$ were prime, it would follow that one of $x_{i_1}-x_{i_2}$, $x_{i_r}$ or $x_{i_1}+\cdots+x_{i_r}$ belongs to $I$, a contradiction. Hence $I$ is not prime.

Now, assume that $G_X$ contains no directed cycle. Then, by
Remark~\ref{rem:def}, every connected component of $G_X$ is a directed
path. Let $i_1\to i_2\to\cdots\to i_r$ be such a path. If $r\ge2$, then
\[
x_{i_1}=z_{i_2},\quad
x_{i_2}=z_{i_3},\quad
\ldots,\quad
x_{i_{r-1}}=z_{i_r}.
\]
Since $i_1$ has indegree zero, no entry among
$x_1,\ldots,x_n$ is equal to $z_{i_1}$. Hence
$z_{i_1}\in\{y_1,\ldots,y_n\}$. Therefore, after ordering the columns as
$i_1,i_2,\ldots,i_r$, the corresponding submatrix of $X$
\[
\begin{bmatrix}
x_{i_1} & x_{i_2} & \cdots & x_{i_r}\\
z_{i_1} & x_{i_1} & \cdots & x_{i_{r-1}}
\end{bmatrix},
\]
can be transformed into $X_r(\{2,\ldots,r\})$ by relabeling the variables.

Repeating this construction for each connected component, we obtain a
decomposition of $X$ into blocks of the form
$X_r(\{2,\ldots,r\})$ and single columns whose second-row entries are
$y$-variables. Equivalently, after a suitable permutation of the columns
and relabeling of the variables, the matrix $X$ is of the form
$X_n(\delta)$ for some $\delta\subseteq\{2,\ldots,n\}$. Therefore, the
ideal generated by the $2$-minors of $X$ is prime by (i) of 
Theorem~\ref{prime+CM+linear}.
\hspace{10.1cm}
\end{proof}

We can now state the main classification theorem.

\begin{Theorem}
\label{mainPCML}
Let $S=K[x_1, \ldots, x_n,y_1, \ldots,y_n]$ denote the polynomial ring in $2n$ variables over a field $K$ and $X$ a $2 \times n$ matrix
\[
X=
\begin{bmatrix}
x_1 & x_2 & \cdots & x_n\\
z_1 & z_2 & \cdots & z_n
\end{bmatrix},
\]
where $z_1,\ldots,z_n$ are pairwise distinct elements of $\{x_1,\ldots,x_n,y_1,\ldots,y_n\}$ and where $z_i\neq x_i$ for all $i=1,\ldots,n$.  Then the following conditions are equivalent:
\begin{enumerate}
\item[(i)] $X$ is a PCML matrix;
\item[(ii)] $I_2(X)$ is prime;
\item[(iii)] the directed graph $G_X$ contains no directed cycle.
\end{enumerate}
\end{Theorem}

\begin{proof}
The implication $(i)\Rightarrow(ii)$ follows immediately from the
definition of a PCML matrix, while $(ii)\Leftrightarrow(iii)$ is
Theorem~\ref{prime}.

It remains to prove $(iii)\Rightarrow(i)$. Suppose that $G_X$ contains
no directed cycle. By Remark~\ref{rem:def}, every connected component of $G_X$
is a directed path. As shown in the proof of Theorem~\ref{prime}, after a
suitable permutation of the columns and relabeling of the variables,
the matrix $X$ is of the form $X_n(\delta)$ for some $\delta\subseteq\{2,\ldots,n\}$. By Theorem~\ref{prime+CM+linear}, the ideal
$I_n(\delta)$ is prime, the quotient $S_\delta/I_n(\delta)$ is
Cohen--Macaulay, and $I_n(\delta)$ has a linear resolution.
These properties are preserved under relabeling of variables and
permutation of columns, and adjoining unused variables does not affect
them. Hence $X$ is a PCML matrix.
\end{proof}

\section*{Acknowledgments} The second author is supported by Scientific and Technological Research Council of Turkey T\"UB\.{I}TAK under the Grant No: 124F113, and is thankful to T\"UB\.{I}TAK for its support.

\bigskip

\begin{footnotesize}
{\bf Declaration of competing interest.}  The authors declare that they have no known competing financial interests or personal relationships that could have appeared to influence the work reported in this paper.

\medskip

{\bf Data availability.} No data was used for the research described in the article. 

\medskip

\end{footnotesize}

\end{document}